\documentclass[a4paper,12pt]{amsart}
\usepackage[left=1in,right=1in,top=1.2in,bottom=1.2in]{geometry}
\usepackage{times}
\usepackage{microtype}
\usepackage{amssymb,url,amscd,mathrsfs,enumerate,mathtools}
\usepackage[all]{xy}
\usepackage{verbatim}
\usepackage[pagebackref,linktocpage]{hyperref}
\hypersetup{
    colorlinks=true,        
    linkcolor=red,          
    citecolor=blue,         
    filecolor=magenta,      
    urlcolor=cyan           
}

\newtheorem{theorem}{Theorem}[section]

\newtheorem{lemma}[theorem]{Lemma}

\newtheorem{conjecture}[theorem]{Conjecture}
\newtheorem{thmx}{Theorem}

\newtheorem{corx}[thmx]{Corollary}

\theoremstyle{definition}
\newtheorem{question}[theorem]{Question}
\newtheorem{definition}[theorem]{Definition}
\newtheorem{example}[theorem]{Example}

\theoremstyle{plain}
\numberwithin{equation}{section}

\theoremstyle{remark}

\DeclareMathOperator{\adeg}{\widehat{\deg}}

\DeclareMathOperator{\Spec}{Spec}

\DeclareMathOperator{\divi}{div}
\DeclareMathOperator{\adivi}{\widehat{\divi}}
\DeclareMathOperator{\vol}{vol}
\DeclareMathOperator{\avol}{\widehat{\vol}}

\DeclareMathOperator{\CC}{\mathbb{C}_{\infty}}

\DeclareMathOperator{\CH}{CH}

\newcommand{\bQ}{{\mathbf Q}}
\newcommand{\oQ}{{\overline{\mathbf Q}}}
\newcommand{\bR}{{\mathbf R}}
\newcommand{\bZ}{{\mathbf Z}}

\newcommand{\cH}{{\mathcal H}}
\newcommand{\cM}{{\mathcal M}}
\newcommand{\cE}{{\mathcal E}}
\newcommand{\cX}{{\mathfrak X}}
\newcommand{\cY}{{\mathcal Y}}
\newcommand{\cZ}{{\mathcal Z}}

\newcommand{\cV}{{\mathcal V}}

\newcommand{\cL}{{\mathcal L}}

\newcommand{\ff}{\mathfrak{f}}

\newcommand{\X}{\mathcal{X}}

\newcommand{\lra}{\longrightarrow}

\renewcommand{\CC}{{\mathbf C}}

\newcommand{\PP}{{\mathbb P}}

\newif\ifhascomments \hascommentstrue
\ifhascomments
  \newcommand{\dragos}[1]{{\color{red}[[\ensuremath{\bigstar\bigstar\bigstar} #1]]}}
  \newcommand{\matt}[1]{{\color{red}[[\ensuremath{\spadesuit\spadesuit\spadesuit} #1]]}}
\else
  \newcommand{\dragos}[1]{}
  \newcommand{\matt}[1]{}
  \newcommand{
    \fi

\newcommand{\ol}[1]{\overline{#1}}

\newcommand{\ratmap}{\dashrightarrow}

\begin{document}

\title{Higher arithmetic degrees of dominant rational self-maps}

\author{Nguyen-Bac Dang}
\address{Institute for Mathematical Sciences, Stony Brook University, Stony Brook, NY 11794-3660, USA}
\email{nguyen-bac.dang@stonybrook.edu}

\author{Dragos Ghioca}
\address{Department of Mathematics, University of British Columbia, 1984 Mathematics Road, Vancouver, BC V6T 1Z2, Canada}
\email{dghioca@math.ubc.ca}

\author{Fei Hu}
\address{Pacific Institute for the Mathematical Sciences, 2207 Main Mall, Vancouver, BC V6T 1Z4, Canada}
\email{hf@u.nus.edu}

\author{John Lesieutre}
\address{Department of Mathematics, The Pennsylvania State University, 204 McAllister Building, University Park, PA 16802, USA}
\email{jdl@psu.edu}

\author{Matthew Satriano}
\address{Department of Pure Mathematics, University of Waterloo, Waterloo, ON N2L 3G1, Canada}
\email{msatrian@uwaterloo.ca}

\begin{abstract}
Suppose that \(f \colon X \ratmap X\) is a dominant rational self-map of a smooth projective variety defined over \({\overline{\mathbf Q}}\).  Kawaguchi and Silverman conjectured that if \(P \in X({\overline{\mathbf Q}})\) is a point with well-defined forward orbit, then the growth rate of the height along the orbit exists, and coincides with the first dynamical degree $\lambda_1(f)$ of $f$ if the orbit of $P$ is Zariski dense in $X$.

In this note, we extend the Kawaguchi--Silverman conjecture to the setting of orbits of higher-dimensional subvarieties of \(X\).  We begin by defining a set of arithmetic degrees of \(f\), independent of the choice of cycle, and we then develop the theory of arithmetic degrees in parallel to existing results for dynamical degrees.  We formulate several conjectures governing these higher arithmetic degrees, relating them to dynamical degrees.
\end{abstract}



\thanks{
	D.G. and M.S. were partially supported by Discovery Grants from NSERC.
	F.H. was partially supported by a UBC-PIMS Postdoctoral Fellowship.
	J.L. was partially supported by NSF grants.
}

\maketitle


\section{Introduction}
\label{section:intro}

\noindent
Suppose that \(f \colon X \ratmap X\) is a rational self-map of a \(d\)-dimensional smooth projective variety defined over \(\oQ\).  Let \(X_f(\oQ)\) denote the set of rational points \(P\) for which the full forward orbit is well-defined.  Fix a Weil height function \(h_X \colon X(\oQ) \to \bR\) and set $h_X^+ \coloneqq \max\{ h_X, 1 \}$.
Following Kawaguchi--Silverman \cite{KS16b}, the \emph{arithmetic degree} of a point \(P \in X_f(\oQ)\) is defined to be the limit
\[
\alpha_f(P) = \lim_{n \to \infty} h_X^+(f^n(P))^{1/n},
\]
supposing that it exists.  This is a measure of the ``arithmetic complexity'' of the orbit of \(P\).

Another basic invariant of \(f\) is its set of dynamical degrees.  For \(0 \leqslant k \leqslant d\), we define
\[
\lambda_k(f) = \lim_{n \to \infty} ((f^n)^\ast H^k \cdot H^{d-k})^{1/n},
\]
where \(H\) is a fixed ample divisor on $X$.  The limit is known to exist (cf.~\cite{DS05,Truong15,Dang}). Also, note that it is independent of the choice of the ample divisor \(H\).  There is expected to be a close relationship between the arithmetic and dynamical degrees: a conjecture of Kawaguchi and Silverman \cite{KS16b} states that if \(P\) has Zariski dense \(f\)-orbit, then \(\alpha_f(P) = \lambda_1(f)\).

Although the theory of dynamical degrees is comparably well-developed, many basic questions remain open about the arithmetic degree $\alpha_f(P)$ of a point $P$ (see \cite{KS14, KS16a, KS16b, matsuzawa, Silverman14, Silverman17, MSS18a, MSS18b, LS} for various results).  The  dynamical degrees of \(f\) satisfy certain inequalities, and there is a well-defined notion of relative dynamical degrees of a rational map preserving a fibration.
We refer to \cite[\S3]{Oguiso14}, \cite[\S4]{DS17}, and references therein for a comprehensive exposition on dynamical degrees, topological and algebraic entropies.

Our aim in this note is to develop a theory of higher arithmetic degrees in parallel to the existing theory for dynamical degrees,  which was suggested by Kawaguchi-Silverman in \cite[Remark 11]{KS16b}.
  We introduce the so-called \(k\)-th arithmetic degree of a self-map as a rough measure of the height growth of \((k-1)\)-dimensional algebraic subvarieties of \(X\), and show that a number of the basic properties of dynamical degrees extend to the arithmetic setting.  Although the arguments are along familiar lines, considerable technical difficulties arise in handling arithmetic intersection numbers. Our arithmetic degrees are independent of the choice of the subvarieties (unlike the \(\alpha_f(P)\) of Kawaguchi--Silverman, which depends on \(P\)).
On the other hand, we also formulate analogous invariants depending on cycles and a higher-dimensional analogue of the Kawaguchi--Silverman conjecture.


\subsection{Examples}
\label{subsec:eg}

Before giving the precise definition, we consider some examples in the first new case: the growth rate of heights of hypersurfaces under self-maps of projective spaces.  In this case, the computations are quite concrete.  Suppose that \(f \colon \PP^d \ratmap \PP^d\) is a dominant rational map.  Fix coordinates on \(\PP^d\).  Given a hypersurface \(V \subset \PP^d\) defined over $\mathbb{Q}$, write out its defining equation in coordinates using coprime integer coefficients.
Let \(c_V\) be the largest of these coefficients (in absolute value) and define the height \(h(V) = \deg(V) + \log(c_V)\).

We will define the arithmetic degree of $V$ by the limit \( \displaystyle \alpha_d(f;V) = \lim_{n \to \infty} h(f^n(V))^{1/n}\); the more general and precise definition appears later as Definition~\ref{def-alpha}.

The conjecture of Kawaguchi and Silverman is that if \(P\) is a point with dense orbit, we have \(\alpha_1(f;P) = \lambda_1(f)\).  Considering some simple two-dimensional examples shows the most obvious analogue of this conjecture is too optimistic: it may not be true that if \(V\) has dense orbit then 
\(\alpha_k(f;V) = \lambda_k(f)\).

\begin{example}
Let \(f \colon \PP^2 \ratmap \PP^2\) be a rational self-map of $\PP^2$ defined in an affine chart by \(f(x,y) = (2y,xy)\).  The iterates are given by \(f^n(x,y) = \left( x^{F_{n-1}} (2y)^{F_n}, x^{F_n} (2y)^{F_{n+1}}/2 \right) \),
where $\{F_n\} = 0,1,1,2,3,5, \ldots$ is the Fibonacci sequence.
The map $f$ is birational, and its inverse is given by the formula \(f^{-1}(x,y) = (2y/x,x/2)\); the iterates of the inverse are given by 
\[
f^{-n}(x,y) = \left( \frac{x^{(-1)^n F_{n+1}}}{(2y)^{(-1)^n F_n}}, \frac{(2y)^{(-1)^n F_{n-1}}}{2x^{(-1)^n F_{n}}} \right).
\]
Clearly, the dynamical degrees of $f$ are given by
\[
\lambda_0(f) = 1, \quad \lambda_1(f) = \frac{1+\sqrt{5}}{2}, \quad \lambda_2(f) = 1.
\]

Consider now the height growth of some subvarieties of \(\PP^2\).  If \(P\) is a point with dense orbit, it is easy to see that \(\alpha_1(f;P) = \frac{1+\sqrt{5}}{2}\).

If \(V\) is a general curve on \(\PP^2\), its image under \(f^n\) can be computed by pulling back its defining equations via the formula for \(f^{-n}\).  From the definition of the height of a hypersurface in \(\PP^d\), we then see that \(\alpha_2(f;V) = \frac{1 + \sqrt{5}}{2}\) as well.

In fact, all arithmetic degrees for this map can be computed as
\[
\alpha_0(f) = 1, \quad \alpha_1(f) = \frac{1+\sqrt{5}}{2}, \quad \alpha_2(f) = \frac{1+\sqrt{5}}{2}, \quad \alpha_3(f) = 1.
\]
\end{example}

\begin{example}
Consider next the map \(f \colon \PP^2 \ratmap \PP^2\) which is given in coordinates by the formula \(f(x,y) = (x^2,y^2)\).  Then \(f^n(x,y) = (x^{2^n},y^{2^n})\).  This map has dynamical degrees
\[
\lambda_0(f) = 1, \quad \lambda_1(f) = 2, \quad \lambda_2(f) = 4.
\]

If \(P\) is a point with dense orbit, then \(\alpha_1(f;P) = 2\).  Consider now the curve \(V \subset \PP^2\) defined by \(x+y = 1\).  Since the curve is parametrized by the map \(\gamma(t) = (t,1-t)\), the strict transform of \(V\) under \(f^n\) is parametrized by \((t^{2^n},(1-t)^{2^n})\).
The arithmetic degrees this time are given by
\[
\alpha_0(f) = 1, \quad \alpha_1(f) = 2, \quad \alpha_2(f) = 4, \quad \alpha_3(f) = 8.
\]
\end{example}

Observe that although in both examples we have \(\alpha_1(f) = \lambda_1(f)\), it is not true in general that \(\alpha_2(f) = \lambda_2(f)\).  These examples, together with a heuristic from the function field case (see also \cite{MSS18b}), naturally lend themselves to the conjecture that \(\alpha_k(f) = \max \{ \lambda_{k}(f),\lambda_{k-1}(f) \} \).

Suppose that \(\oQ(C)\) is the function field of a smooth projective curve \(C\) defined over $\oQ$.  Let \(\pi \colon \X \to C\) be a model of \(X\), i.e., a smooth family for which the generic fiber is isomorphic to \(X\).  Fix an ample divisor \(H_\X\) on \(\X\).
Given an algebraic \(k\)-cycle \(V \subset X\), take \(\cV\) to denote the closure of \(V\) in \(\X\).  The height of \(V\) is then defined to be the intersection number \(\cV \cdot H_\X^{k+1}\), computed on \(\X\). Now, \(f \colon X \ratmap X\) extends to a rational map \(f_{\X} \colon \X \ratmap \X\).  
Suppose further that the closure \(\cV\) is a positive cycle (in a reasonable sense). Then we have
\begin{align*}
\alpha_{k+1}(f;V) &= \lim_{n \to \infty} h_{X}(f^n(V))^{1/n} = \lim_{n \to \infty} (f_\X^n(\cV) \cdot H_\X^{k+1})^{1/n} \\
&= \lim_{n \to \infty} ((f_\X^n)^\ast H_\X^{k+1} \cdot \cV)^{1/n} = \lambda_{k+1}(f_\X).
\end{align*}
However, \(\lambda_{k+1}(f_\X)\) is related to the dynamical degrees of \(f\) itself by the product formula of Dinh--Nguyen \cite{DN11}.  Since \(f_\X\) preserves a fibration over the curve \(C\), we have
\[
\alpha_{k+1}(f;V) = \lambda_{k+1}(f_\X) = \max \{ \lambda_{k+1}(f), \lambda_{k}(f) \}.
\]


\subsection{Definition of higher arithmetic degrees}

We turn at last to the general definitions.  Let $X$ be a smooth projective variety of dimension $d$ defined  over $\bQ$. 
We choose a regular model $\pi \colon \cX \to \Spec\bZ$ of $X$, i.e., a regular projective scheme over $\Spec\bZ$ such that $\pi$ is flat and the generic fiber $\cX_\bQ$ is isomorphic to $X$.
In particular, $\cX$ is an arithmetic variety of relative dimension $d$ over $\Spec\bZ$.
We refer to \cite[\S1 and \S2]{moriwaki00} for a quick tour to the Arakelov intersection theory and various arithmetic positivity properties.
Fix an arithmetically ample line bundle $\ol \cH = (\cH, \| \cdot \|)$ on $\cX$.
The corresponding ample line bundle $\cH_{\bQ}$ on $X$ will be simply denoted by $H$.
For an arithmetic $k$-cycle $\ol Z = (Z, g) \in \widehat{Z}_D^{d+1-k}(\cX)$ on $\cX$, the {\it Faltings height} $h_{\ol \cH}(Z)$ of $Z$ with respect to $\ol\cH$ is defined by the Arakelov intersection
\[
h_{\ol \cH} (Z) \coloneqq \widehat\deg(\ol Z \cdot \widehat c_1(\ol \cH)^k),
\]
where $\widehat c_1(\ol \cH) \in \widehat{\CH}_D^{1}(\cX)$ is the first arithmetic Chern class of $\ol \cH$, and
\[
\widehat\deg \colon \widehat{\CH}_D^{d+1}(\cX) \lra \bR
\]
is the arithmetic degree map.
When $k\ge 1$, let $\deg_\bQ(Z)$ denote the degree of $Z_\bQ$ with respect to the ample line bundle $H$ on $X$.
Namely
\[
\deg_\bQ(Z) \coloneqq (Z_\bQ \cdot H^{k-1}).
\]

Let $\ff \colon \cX \ratmap \cX$ be a dominant rational self-map of $\cX$, and $f$ the corresponding dominant rational self-map of $X$.
Denote by $\pi_1, \pi_2 \colon \Gamma_\ff \to \cX$ the projection from the normalization of the graph of $\ff$ in $\cX \times_\bZ \cX$ onto the first and second component respectively. 

Recall that given an integer $0\leqslant k\leqslant d$, the $k$-th (algebraic) degree of $f$ with respect to $H$ is defined as follows:
\begin{equation*}
\deg_k(f) \coloneqq ( \pi_2^* H^k\cdot \pi_1^* H^{d-k} ). 
\end{equation*}
Then the $k$-th dynamical degree of $f$ is then interpreted as the growth rate of the $k$-th degree of $f$ with respect to any ample divisor, i.e.,
\begin{equation}
\label{eq:lambda}
\lambda_k(f) = \lim_{n\to \infty}\deg_k(f^n)^{1/n}.
\end{equation}
The existence of the limit is a non-trivial result due to \cite{DS05,Truong15,Dang}.
We can also consider analogous versions of these two quantities measuring the degree growth of an algebraic subvariety $V$ of dimension \(k\). Namely, we define the $k$-th degree of $f$ along $V$ as
\[
\deg_k(f;V) \coloneqq ( \pi_2^* H^k\cdot \pi_1^* V)
\]
and the $k$-th dynamical degree of $f$ along $V$ by
\begin{equation}
\label{eq:lambda-V}
\lambda_k(f;V) \coloneqq \limsup_{n \to \infty} \deg_k(f^n;V)^{1/n}.
\end{equation}
Note that we are taking $\limsup$ since in general it is not known whether the limit actually exists.
Similar phenomena also happen in the definition of arithmetic degrees below.

\begin{definition}
\label{def-alpha}
For any integer $0\leqslant k \leqslant d+1$, let $\adeg_k(f)$ denote the $k$-th degree of $f$ with respect to $\ol\cH$, i.e.,
\[
\adeg_k(f) \coloneqq \adeg \big( \widehat c_1(\pi_2^*\ol\cH)^k \cdot \widehat c_1(\pi_1^*\ol\cH)^{d+1-k} \big).
\]
Then the {\it $k$-th arithmetic degree} of $f$ is defined as follows:
\begin{equation}
\label{eq:alpha}
\alpha_k(f) \coloneqq \limsup_{n\to \infty} \adeg_k(f^n)^{1/n}.
\end{equation}

Similarly, let $\cV$ denote the closure of the above $V$ in $\cX$.
We define
\[
\adeg_{k+1}(f; V) \coloneqq \adeg \big( \widehat c_1(\pi_2^* \ol \cH)^{k+1} \cdot \pi_1^*\cV \big),
\]
and
\begin{equation}
\label{eq:alpha-V}
\alpha_{k+1}(f; V) \coloneqq \limsup_{n\to \infty} \adeg_{k+1}(f^n; V)^{1/n}.
\end{equation}
\end{definition}

When \(f\) is a morphism, it is not necessary to take a resolution when defining \(\deg_k\) and \(\adeg_k\).
In this case, we have
\[
\deg_k(f; V) = \pi_2^\ast H^k \cdot \pi_1^\ast V = H^k \cdot f_*V = \deg_k(f_*V), \text{ and }
\]
\[
\adeg_{k+1}(f; V) = \adeg \big( \widehat c_1(\pi_2^* \ol \cH)^{k+1} \cdot \pi_1^*\ol\cV \big) = \adeg \big( \widehat c_1(\ol \cH)^{k+1} \cdot \ff_*\cV \big) = h_{\ol\cH}(f_*V).
\]
Thus \(\lambda_k(f; V)\) (respectively, \(\alpha_{k+1}(f; V)\)) measure the degree (respectively, the height growth of the iterates) of an algebraic subvariety.


Here \(f_\ast V\) denotes the cycle-theoretic pushforward of \(V\), not just its strict transform; if \(f\) is ramified along \(V\), this class \(V\) multiplied by the degree of the ramification.  Subtleties can arise when some iterate(s) of \(f\) is ramified along \(V\), and when infinitely many different iterates are ramified along \(V\) the ramification is reflected in the arithmetic degree of the cycle: \(f^n_\ast V\) is equal to the class of the strict transform \(f^n(V)\) with a factor of the ramification degree.  Consider for contrast the following examples.

\begin{example}
Define \(f \colon \PP^2 \ratmap \PP^2\) by setting \(f(x,y) = (x^2,y)\), and let \(V\) be the curve defined by \(y = 0\).
Then \(V\) is invariant under \(f\) and \(f\vert_V\) is the squaring map on \(\PP^1\).  As a result, \(f_*^n(V) = 2^n V\), and we have \(\alpha_2(f;V) = 2\).

On the other hand, let \(W\) be the curve defined by \(x=y\); this curve is equivalent to \(V\) as a cycle and its iterates are also disjoint from the indeterminacy point \([0:1:0]\).  Then \(f^n(W)\) is defined by the equation \(x = y^{2^n}\), so that \(h(f^n_\ast(W)) = 2^n\).  Thus we find that \(\alpha_2(f;W) = 2\) as well.
\end{example}

It is also possible that \(f\) is ramified along a curve only a finite number of times, in which the extra factor corresponding to ramification disappears in the limit of \(h(f^n_\ast (V))^{1/n} \).

It is also conceivable that \(f\) could ramify along \(V\) infinitely many times, but with such infrequency that the ramification is not reflected in the arithmetic degree.  This is impossible, assuming the following special case of a generalization of the dynamical Mordell--Lang conjecture (see \cite{BGT16} and also \cite{BSS17, LL19}).

\begin{question}
Suppose that \(f \colon X \ratmap X\) is a rational map and that \(V \subset X\) is a subvariety with well-defined forward orbit.  Is
\[
R(f,V) = \left\{n : f^n(V) \subseteq \operatorname{Ram}(f) \right\}
\]
the union of a finite set and finitely many arithmetic progressions?
\end{question}

\subsection{Conjectures and results}

As we have mentioned before, it is known that the limit \eqref{eq:lambda} defining \(\lambda_k(f)\) indeed exists; in the other three cases \eqref{eq:lambda-V}-\eqref{eq:alpha-V}, this is unknown and it is necessary to take the \(\limsup\) in the definition.  We will show later that the limit defining \(\alpha_1(f)\) actually exists and is equal to \(\lambda_1(f)\).
We expect that the limit exists for the other \(\alpha_k(f)\) as well.  It is a more subtle question whether the limits exist for the invariants which depend on a particular cycle and not just the map \(f\).  In the case of \(\alpha_1(f;P)\), the existence of the limit is one part of Kawaguchi--Silverman conjecture. Even in the case of \(\lambda_k(f;V)\) (a question with no arithmetic content), this seems to be an open question.

\begin{conjecture}
Let \(f \colon X \ratmap X\) be a dominant rational map, and let \(V\) be a subvariety of dimension \(k\).  Then the limits defining \(\lambda_k(f;V)\), \(\alpha_k(f)\), and \(\alpha_{k+1}(f;V)\) exist and are independent of the choices of \(H\) and \(\ol \cH\) as well as the model \(\cX\).
\end{conjecture}

We also expect that \(\alpha_{k+1}(f)\) provides an upper bound on the growth rate of the heights of algebraic subvarieties of dimension $k$.

\begin{conjecture}
\label{conj:higher-KSC}
Let \(f \colon X \ratmap X\) be a dominant rational map, and let \(V\) be a subvariety of dimension \(k\).  Then \(\alpha_{k+1}(f;V) \leqslant \alpha_{k+1}(f)\).  If \(V\) has Zariski dense orbit, then \(\alpha_{k+1}(f;V) = \alpha_{k+1}(f)\).
\end{conjecture}

Although we are not able to prove that the $\limsup$ is actually a limit independent of any choice when $k\geqslant 2$,
recent work due to Ikoma \cite[Theorem 2.9]{ikoma} based on the arithmetic Hodge index theorem of Yuan--Zhang \cite{yuan_zhang_hodge_index} shows that the arithmetic intersection shares many similarities with the algebraic intersection.
This allows us to extend a number of basic properties of dynamical degrees to the arithmetic setting.
For example, it is shown that the Khovanski--Teissier inequalities hold in arithmetic intersection theory, and the log concavity of the arithmetic degrees follows.  We list some basic properties of the arithmetic degrees in the next theorem.

\begin{thmx}
\label{thmA}
Using notation as above, the following assertions hold.
\begin{enumerate}[{\em (1)}]
\item The sequence $k \mapsto \alpha_k(f)$ is log-concave.
\item If \(f\) is birational, then \(\alpha_k(f) = \alpha_{d+1-k}(f^{-1})\).
\end{enumerate}
\end{thmx}


We also formulate an analogue of the Kawaguchi--Silverman conjecture, which predicts that the height growth of cycles is controlled by dynamical degrees.  In this case, the natural prediction derives from the case of function fields (see Section~\ref{subsec:eg}), in which it reduces to the familiar Dinh--Nguyen product formula for dynamical degrees (cf.~\cite{DN11}).

\begin{conjecture}
\label{conj:product-formula}
Let \(f\) be a dominant rational self-map of \(X\).  Then for any integer \( 1 \leqslant k \leqslant d\), one has
\[
\alpha_k(f) = \max \{ \lambda_k(f), \lambda_{k-1}(f) \}.
\]
\end{conjecture}

We prove this conjecture in the case that \(\dim X = 2\) and \(f\) is birational (see Theorem~\ref{thm:surface}) and in the case that $f \colon X \to X$ is a polarized endomorphism (see Theorem~\ref{thm:polarized}).

Our next result provides an inequality of Conjecture~\ref{conj:product-formula}.

\begin{thmx}
\label{thmB}
Let \(f\) be a dominant rational self-map of \(X\).  Then for any integer \( 1 \leqslant k \leqslant d\), one has
\begin{equation*}
\alpha_k(f) \geqslant \max \{ \lambda_k(f), \lambda_{k-1}(f) \}.
\end{equation*}
\end{thmx}
 
Observe that combining Conjectures \ref{conj:higher-KSC} and \ref{conj:product-formula} provides a precise prediction about the growth rate of heights of iterates \(f^n(V)\) for \(V\) with dense orbit:
\[
\alpha_{k+1}(f;V) = \alpha_{k+1}(f) = \max \{ \lambda_{k+1}(f),\lambda_{k}(f) \}.
\]
We have seen in Section~\ref{subsec:eg} that either of these two numbers can appear, even in the case that \(V\) is a curve on a surface.  These conjectures should be understood as the higher-dimensional formulation of the Kawaguchi--Silverman conjecture.

The proof of Theorem \ref{thmB} relies on an arithmetic Bertini argument (see Lemma \ref{lem_expansion}) which allows us to relate the arithmetic degree with a height of a given particular cycle. Then one applies an inequality due to Faltings \cite{faltings_diophantine} to bound below the height of the cycles by its algebraic degree.

In the particular case where $k=1$, we show that the sequence $\adeg_1(f^n)$ is submultiplicative.

\begin{thmx}
\label{thmC}
Let $X$, $Y$, $Z$ be smooth projective varieties defined over $\bQ$ of dimension $d$.
Choose the regular models $\cX$, $\cY$ and $\cZ$ of $X,Y$ and $Z$ respectively.
Fix three arithmetically ample line bundles $\ol \cH_\cX$, $\ol \cH_\cY$, and $\ol \cH_\cZ$ on $\cX$, $\cY$ and $\cZ$ respectively.  
Then there exists a constant $C>0$ such that for any dominant rational maps $f \colon X \ratmap Y$ and $g \colon Y \ratmap Z$, one has
\[
\adeg_1( g\circ f) \leqslant C \adeg_1(f) \adeg_1(g),
\]
where the intersection numbers are taken with respect to the chosen arithmetically ample line bundles on each model.
\end{thmx}

The above result relies heavily on Yuan's work \cite{yuan_volume} on the properties of the arithmetic volume (see also \cite{chen_differentiability}). From Yuan's arithmetic volume estimates, one deduces an effective way to obtain an arithmetically pseudo-effective line bundle. 
The above theorem yields the following consequence.

\begin{corx}
\label{corD}
The first arithmetic degree $\alpha_1(f)$ of $f$ is equal to the limit:
\begin{equation*}
\alpha_1(f) = \lim_{n \rightarrow \infty} \adeg_1(f^n)^{1/n},
\end{equation*}
it is a finite number and independent of the choices of the regular model $\cX$ of $X$, the birational model of $X$, or the arithmetically ample line bundle on $\cX$. 
\end{corx}

More precisely, Theorem~\ref{thmA} shows that the growth rate of the sequence $\adeg_1(f^n)$ is also a birational invariant.
In this particular case, the relationship between the arithmetic degree and the dynamical degree is  more constrained.

\begin{thmx}
\label{thmE}
Let $f$ be a dominant rational self-map of $X$. Then one has
\begin{equation*}
\alpha_1(f)=\lambda_1(f).
\end{equation*}
\end{thmx}

This last result together with Theorem~\ref{thmA} implies that every arithmetic degree $\alpha_k(f)$ is finite, precisely:
\begin{equation*}
\alpha_k(f) \leqslant \lambda_1(f)^k.
\end{equation*}

Let us explain how one can understand the above theorem.
Since Theorem~\ref{thmB} proves that $\lambda_1(f) \leqslant \alpha_1(f)$, the proof of this last result amounts in proving the converse inequality $\alpha_1(f) \leqslant \lambda_1(f)$. 
Using an arithmetic Bertini argument (see Section~\ref{section_technical}) and a generalization of Siu's analytic estimates in higher codimension due to Xiao \cite{xiao}, we prove a more general statement (see Theorem \ref{thm_upper_bound}), namely that for all integer $n$, there exists an algebraic cycle $Z$ of dimension $k-1$ on $X$ such that
$$\adeg_k(f^n) \leqslant h_{\ol \cH}(f^n Z) + C \deg_k(f^n)$$
for a positive constant $C>0$ which is independent of $p$ and $Z$.
In the case where $k=1$, the degree of $f^n$ with respect to $\ol\cH$ is controlled by the height of a point and the algebraic degree of $f^n$.
We finally conclude using the fact that the height of the images of a rational point by $f^n$ is controlled by the first dynamical degree of $f$ (see \cite{KS16b,matsuzawa}).

\subsection*{Acknowledgements}

The first, fourth and fifth authors thank Yohsuke Matsuzawa and Joe Silverman for helpful discussions. The first author also thanks Charles Favre, Thomas Gauthier and Mattias Jonsson for their suggestions and remarks. Part of this work was done during the Simons Symposium: Algebraic, Complex and Arithmetic Dynamics in May 2019 at Schloss Elmau. We are grateful to the Simons foundation for its generous support.

\section{On arithmetic degrees of codimension \texorpdfstring{$k$}{k}}

\subsection{Proof of Theorem \ref{thmA}}

Fix a model $\cX$ of $X$ and an arithmetically ample line bundle $\ol \cH$ on $\cX$. We take $f \colon X \ratmap X$ a dominant rational self-map of $X$.
By abuse of notation, we also denote by $f$ the dominant rational map on the regular model $\cX$.
Let $\pi_{1n},\pi_{2n}$ be the projections from the normalization of the graph of $f^n$ in $\cX \times_\bZ \cX$ to the first and second component respectively.
By \cite[Theorem~2.9(1)]{ikoma} applied the arithmetic $\pi_{1n}^* \ol \cH$, $\pi_{2n}^* \ol \cH$, we have:
\[
\quad \ \adeg_k(f^n)^2 = \adeg \big( \widehat c_1(\pi_{2n}^*\ol\cH)^k \cdot \widehat c_1(\pi_{1n}^*\ol\cH)^{d+1-k} \big)^2 \geqslant 
\]
\[
\adeg \big( \widehat c_1(\pi_{2n}^*\ol\cH)^{k+1} \cdot \widehat c_1(\pi_{1n}^*\ol\cH)^{d-k} \big) \adeg \big( \widehat c_1(\pi_{2n}^*\ol\cH)^{k-1} \cdot \widehat c_1(\pi_{1n}^*\ol\cH)^{d+2-k} \big).
\]
We thus obtain that
\begin{equation*}
\adeg_k(f^n)^2 \geqslant \adeg_{k+1}(f^n) \adeg_{k-1}(f^n).
\end{equation*}
Taking the $n$-th root as $n$ tends to infinity yields that
\begin{equation*}
\alpha_k(f)^2 \geqslant \alpha_{k-1}(f) \alpha_{k+1}(f)
\end{equation*}
and the sequence $k \mapsto \alpha_k(f)$ is locally log-concave at each point.
It is thus globally log-concave, as required.

The second assertion immediately follows from the definition of $\alpha_k(f)$, with the roles of $\pi_1$ and $\pi_2$ reversed.
\qed

\subsection{A technical lemma}
\label{section_technical}

In this section, we prove a technical lemma which relates the two different $k$-th degrees of a rational self-map $f$ as well as the height of certain points on the $f$-orbit. To do so, let us introduce a few notations.

We fix a dominant rational map $f \colon X \ratmap X$ defined over $\bQ$.
Denote by $\pi_1, \pi_2$ the projections from a generically smooth birational model $\Gamma_f$ of the graph of $f$ in $\cX \times \cX$ onto the first and the second factor respectively. 
We fix an arithmetically ample line bundle $\ol \cH$ on $\cX$ and an integer $k \leqslant d$.
Take a sequence of sections $s_{1,\epsilon}, \ldots , s_{d, \epsilon}$ of an arithmetically ample line bundle such that $\adivi(s_{i,\epsilon}) = \ol \cH_\epsilon$ is arbitrarily close to $\pi_1^* \ol \cH$ and whose norm is uniformly bounded by $1$. We denote by $\omega_{1,\epsilon}$ the smooth $(1,1)$-forms associated to the class $\adivi(s_{i,\epsilon})$.
For an integer $1 \leqslant i \leqslant d+1-k$, denote by $Z_{i,\epsilon}$ the smooth subvarieties given by:
\begin{equation*}
Z_{i, \epsilon} = \divi(s_{i,\epsilon}) \cdot \ldots \cdot \divi(s_{d+1-k, \epsilon}).
\end{equation*}
Observe that the support of the cycles $Z_{i,\epsilon}$ forms an increasing sequence of closed subvarieties of $\Gamma_f$ where $Z_{d+1-k,\epsilon}= \divi(s_{d+1-k,\epsilon})$.
For the next proof, we will also set $Z_{d+2-k,\epsilon} \coloneqq \Gamma_f$.  
   
\begin{lemma} \label{lem_expansion} One has:
\begin{equation*}
\adeg_k(f) = \lim_{\epsilon \rightarrow 0} \adeg \big(\widehat c_1(\pi_2^* \ol \cH)^k \cdot Z_{1,\epsilon} \big) -  \sum_{j=2}^{d+2-k} \int_{Z_{j,\epsilon}(\CC)} \log(||s_{j-1,\epsilon}||) \pi_2^* \omega_X^k \wedge \omega_{1,\epsilon}^{j-2}.
\end{equation*}
\end{lemma}
\begin{proof}
Let us write $\adeg_k(f)$ as follows:
\begin{equation*}
\adeg_k(f) = \adeg \big(\widehat c_1(\pi_2^* \ol \cH)^k \cdot \adivi(s_{1,\epsilon}) \cdot \ldots \cdot \adivi(s_{d+1-k,\epsilon}) \big). 
\end{equation*}
By \cite[Formula 2.3.8 of Proposition 2.3.1(vi)]{bost_gillet_soule} applied to the arithmetic divisor $\adivi(s_{d+1-k})$, we have:
\begin{align*}
\adeg_k(f) &= \lim_{\epsilon \rightarrow 0} \adeg \big(\widehat c_1(\pi_2^* \ol \cH)^k \cdot \adivi(s_{1,\epsilon}) \cdot \ldots \cdot \adivi(s_{d-k,\epsilon}) \cdot \divi(s_{d+1 - k ,\epsilon}) \big) \\
& \quad - \int_{\Gamma_f(\CC)} \log(|| s_{d+1-k,\epsilon}||) \pi_2^* \omega_X^k \wedge \omega_{1,\epsilon}^{d-k}\\
& = \lim_{\epsilon \rightarrow 0} \adeg \big(\widehat c_1(\pi_2^* \ol \cH)^k \cdot \adivi(s_{1,\epsilon}) \cdot \ldots \cdot \adivi(s_{d-k,\epsilon}) \cdot Z_{d+1-k,\epsilon} \big) \\
& \quad - \int_{\Gamma_f(\CC)} \log(|| s_{d+1-k,\epsilon}||) \pi_2^* \omega_X^k \wedge \omega_{1,\epsilon}^{d-k}.
\end{align*}
We apply inductively \cite[Formula 2.3.8 of Proposition 2.3.1(vi)]{bost_gillet_soule} as follows:
\[
\adeg \big( \widehat c_1(\pi_2^* \ol \cH)^k \cdot \widehat c_1(\ol \cH_\epsilon)^{d+1-k}\cdot [X] \big) = 
\]
\[
\adeg \big( \widehat c_1(\pi_2^* \ol \cH)^k \cdot \widehat c_1(\ol \cH_\epsilon)^{d-k} \cdot Z_{d+1-k,\epsilon} \big) - \int_{X(\CC)} \log(|| s_{d+1-k,\epsilon}||) \pi_2^* \omega_X^k \wedge \omega_{1,\epsilon}^{d-k},
\]
\[
\adeg \big( \widehat c_1(\pi_2^* \ol \cH)^k \cdot \widehat c_1(\ol \cH_\epsilon)^{d-k} \cdot Z_{d+1-k,\epsilon} \big) = 
\]
\[
\adeg \big( \widehat c_1(\pi_2^* \ol \cH)^k \cdot \widehat c_1(\ol \cH_\epsilon)^{d-k-1} \cdot Z_{d-k,\epsilon} \big) - \int_{Z_{d+1-k,\epsilon}(\CC)} \log(|| s_{d-k,\epsilon}||) \pi_2^* \omega_X^k \wedge \omega_{1,\epsilon}^{d-k-1},
\]
\[
\adeg \big( \widehat c_1(\pi_2^* \ol \cH)^k \cdot \widehat c_1(\ol \cH_\epsilon)^{d-k-1} \cdot Z_{d-k,\epsilon} \big) =
\]
\[
\adeg \big( \widehat c_1(\pi_2^* \ol \cH)^k \cdot \widehat c_1(\ol \cH_\epsilon)^{d-k-2} \cdot Z_{d-k-1,\epsilon} \big) - \int_{Z_{d-k,\epsilon}(\CC)} \log(|| s_{d-k-1,\epsilon}||) \pi_2^* \omega_X^k \wedge \omega_{1,\epsilon}^{d-k-2},
\]
\[\vdots\]
\[
\adeg \big( \widehat c_1(\pi_2^* \ol \cH)^k \cdot  \widehat c_1(\ol \cH_\epsilon)^{2} \cdot Z_{3,\epsilon} \big) =
\]
\[
\adeg \big( \widehat c_1(\pi_2^* \ol \cH)^k \cdot  \widehat c_1(\ol \cH_\epsilon) \cdot Z_{2,\epsilon} \big) -\int_{Z_{3,\epsilon}(\CC)} \log(|| s_{2,\epsilon}||) \pi_2^* \omega_X^k \wedge \omega_{1,\epsilon},
\]
\[
\adeg \big( \widehat c_1(\pi_2^* \ol \cH)^k \cdot  \widehat c_1(\ol \cH_\epsilon) \cdot Z_{2,\epsilon} \big) = \adeg \big( \widehat c_1(\pi_2^* \ol \cH)^k \cdot Z_{1,\epsilon} \big) -\int_{Z_{2,\epsilon}(\CC)} \log(|| s_{1,\epsilon}||) \pi_2^* \omega_X^k.
\]

From the above sequence of equalities, we deduce that
\begin{equation*}
\adeg_k(f) = \lim_{\epsilon \rightarrow 0} \adeg \big(\widehat c_1(\pi_2^* \ol \cH)^k \cdot Z_{1,\epsilon} \big) - \sum_{j=2}^{d+2-k} \int_{Z_{j,\epsilon}(\CC)} \log(||s_{j-1,\epsilon}||) \pi_2^* \omega_X^k \wedge \omega_{1,\epsilon}^{j-2} .
\end{equation*}
This finishes the proof of the lemma.
\end{proof}

\subsection{Proof of Theorem \ref{thmB}}

We shall use the following inequality due to Faltings \cite[Proposition~2.16]{faltings_diophantine}:

There exists a constant $C>0$ depending only on the arithmetically ample line bundle $\ol \cH$ on $\cX$ (depending on the supremum of the norm of a section of the corresponding line bundle at infinite places) such that:
\begin{equation*}
\label{eq_faltings}
h_{\ol \cH} (Z) \geqslant C ( H^k \cdot Z),
\end{equation*}
for any effective  cycle $Z \in Z^{d-k}(\cX)$. 
By Lemma \ref{lem_expansion}, for any integer $p $, we have:
\begin{equation*}
\adeg_k(f^p) = \lim_{\epsilon \rightarrow 0} h_{\ol \cH} (f^p {\pi_1}_*(Z_{1,\epsilon})) - \sum_{j=2}^{d+2-k} \int_{Z_{j,\epsilon}(\CC)} \log(||s_{j-1,\epsilon}||) \pi_2^* \omega_X^k \wedge \omega_{1,\epsilon}^{j-2},
\end{equation*}
where $Z_{1,\epsilon}$ and the $s_{i,\epsilon}$ are defined in Section~\ref{section_technical}. 
Observe that this cycle depends on each iterate $p$ since it a cycle on the graph of $f^p$.
Since the supremum of the norm of the sections $s_{j,\epsilon}$ is bounded by $1$, all the terms in the integral are non-negative.
We thus have
\begin{equation*}
\adeg_k(f^p) \geqslant h_{\ol \cH} (f^p {\pi_1}_* (Z_{1,\epsilon})).
\end{equation*}
Using the Faltings inequality, we thus deduce $\adeg_k(f^p) \geqslant C (H^{k-1} \cdot Z_{1,\epsilon})$. 
Taking the $p$-th root as $p$ tends to infinity finally gives:
\begin{equation*}
\alpha_k(f) \geqslant \lambda_{k-1}(f).
\end{equation*}
For the other inequality, we observe that the roles played by the line bundles $\pi_1^* \ol \cH$ and $\pi_2^* \ol \cH$ are symmetric and the technical lemma also proves that $\adeg_k(f^p) \geqslant C h_{\ol \cH} (f^{-p} Z)$ for an appropriate complete intersection cycle of dimension $k$ and we conclude that $\alpha_k(f) \geqslant \lambda_k(f)$. 
We have thus proven that $\alpha_k(f) \geqslant \max \{ \lambda_k(f), \lambda_{k-1}(f) \}$, as required.
\qed

We have conjectured that in fact the inequality of Theorem~\ref{thmB} is always an equality (see Conjecture~\ref{conj:product-formula}).
There are two cases in which we can prove this conjecture.

\begin{theorem}
\label{thm:surface}
If \(f \colon X \ratmap X\) is a birational self-map of a surface, then
\[
\alpha_k(f) = \max \{ \lambda_k(f),\lambda_{k-1}(f) \}
\]
for all \(1 \leqslant k \leqslant d\).  
\end{theorem}

\begin{proof}
It is always true that \(\alpha_0(f) = 1\), while \(\alpha_1(f) = \lambda_1(f)\) for any map according to Theorem~\ref{thmE}.  We have
\[
\alpha_2(f) = \alpha_1(f^{-1}) = \lambda_1(f^{-1}) = \lambda_1(f),
\]
and since \(\lambda_2(f) = 1\) as \(f\) is birational, we have \(\alpha_2(f) = \max \{ \lambda_1(f),\lambda_2(f) \} \) by definition.  
\end{proof}

Recall that a surjective self-morphism $f \colon X \to X$ is polarized, if there exists an ample line bundle $L$ on $X$ such that $f^*L \simeq L^{q}$ for some integer $q>1$.

\begin{theorem}
\label{thm:polarized}
If \(f \colon X \to X\) is a polarized endomorphism with respect to the ample line bundle $H=\cH_\bQ$ on $X$, then
\[
\alpha_k(f) = \max \{ \lambda_k(f),\lambda_{k-1}(f) \}
\]
for all \(1 \leqslant k \leqslant d\). 
\end{theorem}

\begin{proof}
We fix an isomorphism $\phi \colon \cH^q \simeq f^*\cH$ on the arithmetic model $\cX$ of $X$,
where $\ol\cH = (\cH, || \cdot ||)$ is a fixed arithmetically ample line bundle on $\cX$.
By \cite[Theorem 2.2]{zhang_small_points}, there exists a unique metric $\| \cdot \|_0$ on $\cH$ satisfying:
\begin{equation*}
|| \cdot ||_0 = (\phi^* f^* || \cdot ||_0)^{1/q}.
\end{equation*}
This metric is obtained from the initial metric $|| \cdot ||$ on $\cH$ using Tate's limiting argument. As the initial $\ol \cH$ is arithmetically ample, so is the new arithmetic line bundle $\ol \cH_0 \coloneqq (\cH, || \cdot||_0)$.

Let us compute the $k$-th degree of $f^n$ with respect to the choice of the arithmetically ample line bundle $\ol \cH_0$:
\begin{equation*}
\adeg_k(f^n) = \adeg \big( \widehat{c}_1((f^n)^* \ol \cH_0)^k \cdot \widehat{c}_1(\ol \cH_0)^{d+1-k} \big) = q^{kn} \adeg \big( \widehat{c}_1(\ol \cH_0)^{d+1} \big).
\end{equation*}
This proves that $\alpha_k(f) = q^k$.
On the other hand, $\lambda_{k}(f) = q^{k}$ for any $k$.
We hence conclude that
\begin{equation*}
\alpha_k(f) = \max\{\lambda_k(f), \lambda_{k-1}(f)\},
\end{equation*}
as required.
\end{proof}

\subsection{Upper bound for the arithmetic degree}
We now give an upper bound for the $k$-th degree in the following result.
To do so, we will need the following generalization of Siu's analytic estimates in higher codimension due to Xiao \cite[Remark 3.1]{xiao}.

\begin{theorem}[{cf.~Xiao \cite{xiao}}] Take $\omega_1, \omega_2$ two big nef $(1,1)$ classes on a smooth compact K\"ahler manifold $X$ of dimension $d$. 
Then for any integer $k \leqslant n$,  the $(k,k)$-form
\[
\binom{d}{k} \dfrac{\int_X \omega_1^k \wedge \omega_2^{d-k}}{\int_X \omega_2^d} \omega_2^k - \omega_1^k
\]
is represented by a closed positive current.
\end{theorem}

We now state our next result, which gives an upper bound of the $k$-th degree of $f^p$. 

\begin{theorem}
\label{thm_upper_bound}
Let $f \colon X \dashrightarrow X$ be a dominant rational map on a projective variety defined over $\bQ$. 
There exists a constant $C>0$ such that for any integers $k\leqslant d+1$, for any integer $p$, there exists a cycle $Z$ of dimension $k-1$   whose support is not contained in  the indeterminacy locus of $f^p$ such that:
\begin{equation*}
\adeg_k(f^p) \leqslant h_{\ol \cH}(f^p Z) + C \deg_k(f^p).
\end{equation*}
\end{theorem}

\begin{proof}
We adopt the same notation of Section~\ref{section_technical}. 
Fix an integer $p$, we consider a generically smooth birational model $\Gamma$ of the graph of $f^p$.
By Lemma \ref{lem_expansion}, we have:
\begin{equation*}
\adeg_k(f) = \lim_{\epsilon \rightarrow 0} \adeg\big( \widehat c_1(\pi_2^* \ol \cH)^k \cdot Z_{1,\epsilon} \big) - \sum_{j=2}^{d+2-k} \int_{Z_{j,\epsilon}(\CC)} \log(||s_{j-1,\epsilon}||) \pi_2^* \omega_X^k \wedge \omega_{1,\epsilon}^{j-2}.
\end{equation*}
For each integer $2 \leqslant j \leqslant d+2-k $, Xiao's inequality implies that the $(k,k)$-form
\begin{equation*}
\binom{j+k-2}{k} \dfrac{\int_{Z_{j,\epsilon}(\CC)} \pi_2^* \omega_X^k \wedge \pi_1^* \omega_X^{j-2}}{\int_{Z_{j,\epsilon}(\CC)} \pi_1^* \omega_X^{k+j-2}} \pi_1^* \omega_X^k - \pi_2^*\omega_X^k
\end{equation*}
is represented by a closed positive current on $Z_{j,\epsilon}(\CC)$.
Since $-\log(||s_{j-1,\epsilon}||) \omega_{1,\epsilon}^{j-2}$ is a positive $(j-2,j-2)$-form on $Z_{j,\epsilon}(\CC)$, we deduce that
\begin{align*}
&-\int_{Z_{j,\epsilon}(\CC)} \log(||s_{j-1,\epsilon}||) \pi_2^* \omega_X^k \wedge \omega_{1,\epsilon}^{j-2} \\
\leqslant & -C \dfrac{\int_{Z_{j,\epsilon}(\CC)} \pi_2^* \omega_X^k \wedge \pi_1^* \omega_X^{j-2} }{ \int_{Z_{j,\epsilon}(\CC)} \pi_1^* \omega_X^{k+j-2}} \int_{Z_{j,\epsilon}(\CC)} \log(||s_{j-1,\epsilon}||)\pi_1^* \omega_X^k \wedge \pi_1^* \omega_{1,\epsilon}^{j-2},
\end{align*}
where $C = \dfrac{(j+k-2)!}{k! (j-2)!}$.

Observe that the term on the right hand side of the above inequality is bounded by 
$ C' \deg_k(f^p) $ as $\epsilon $ goes to zero where $C' $  is a positive constant bounded above by $C \avol(\ol \cH)/ \vol(H)$. 
We thus obtain:
\begin{equation*}
\adeg_k(f^p) \leqslant h_{\ol \cH}(f^p {\pi_1}_*(Z_{1,\epsilon})) + C' \deg_k(f^p),
\end{equation*}
as required.
\end{proof}

\section{On the first arithmetic degree}

\subsection{Proof of Theorem \ref{thmC}}

Our proof follows closely the recent algebraic approach for the existence of dynamical degrees for dominant rational self-maps due to Dang \cite{Dang}.

We shall use the following inequality on arithmetic divisors  due to Yuan \cite{yuan_volume}.

\begin{theorem}[{\cite[Theorem B]{yuan_volume}}]
Let $\ol\cL$, $\ol\cM$ and $\ol\cE$ be three Hermitian line bundles over an arithmetic variety $\cX$ of relative dimension $d$. Assume that $\ol\cL$ and $\ol\cM$ are arithmetically ample. Then
\[
\avol( \ol\cL + \ol\cM)^{1/(d+1)} \geq \avol(\ol\cL)^{1/(d+1)} + \avol(\ol\cM)^{1/(d+1)}. 
\]
\end{theorem}

An important consequence of the above fact (see \cite[Remark,  p.~1459]{yuan_volume}) is that the class $\ol\cL - \ol\cM$ is arithmetically big if
\[
\adeg \big( \widehat c_1(\ol\cL)^{d+1} \big) > (d+1) \adeg \big( \widehat c_1(\ol\cL)^{d} \cdot \widehat c_1(\ol\cM) \big).
\]
This inequality is the key ingredient to prove the submultiplicativity of the sequence $(\adeg_1(f^p))$. 

\begin{proof}[Proof of Theorem~\ref{thmC}]
Fix $X, Y , Z$ three smooth projective varieties of dimension $d$ defined over $\bQ$.
Let $f \colon X\dashrightarrow Y$ and $g \colon Y\dashrightarrow Z$ be two dominant rational maps.
Fix three regular models $\cX,\cY$ and $\cZ$ of $X,Y$ and $Z$ respectively and three arithmetically ample divisors $\ol \cH_\cX$, $\ol \cH_\cY$ and $\ol \cH_\cZ$ respectively.
Take $\Gamma_f,\Gamma_g$ the normalization of the graphs of $f \colon \cX \dashrightarrow \cY$ and $g \colon \cY \dashrightarrow \cZ$ respectively.
Consider also the normalization $\Gamma$ of the graph of the rational map $\Gamma_f \dashrightarrow \Gamma_g$ induced by $f$.
We also denote by $\pi_1, \pi_2$ the projection of $\Gamma_f$ onto the $\cX$ and $\cY$, $u,v$ the projection of $\Gamma$ onto $\Gamma_f$ and $\Gamma_g$, and by $\pi_3, \pi_4$ the projections of $\Gamma_g$ onto $\cY$ and $\cZ$ respectively. 
We thus obtain the following diagram.
\[
\xymatrix{  && \Gamma \ar[ld]_u \ar[rd]^v  & & \\
 & \Gamma_{f} \ar@{-->}[rr] \ar[ld]_{\pi_1} \ar[rd]^{\pi_2}  & & \Gamma_{g}  \ar[ld]_{\pi_3} \ar[rd]^{\pi_4} & \\
 \cX  \ar@{-->}[rr]_{f} & &  \cY \ar@{-->}[rr]_{g} & & \cZ.
}
\]

We now apply Yuan's estimates to 
$\ol \cL = u^* \pi_2^* \ol \cH_\cY$ and $\ol \cM = v^* \pi_4^* \ol \cH_\cZ$, the class
\begin{equation*}
(d+1) \dfrac{\adeg \big( \widehat c_1 (u^* \pi_2^* \ol \cH_\cY)^d \cdot \widehat c_1 (v^* \pi_4^* \ol \cH_\cZ) \big)}{\adeg \big( \widehat c_1 (u^* \pi_2^* \ol \cH_\cY)^{d+1} \big)} u^* \pi_2^* \ol \cH_\cY - v^* \pi_4^* \ol \cH_\cZ
\end{equation*}
is pseudo-effective.
Thus the intersection with the arithmetically nef class $u^* \pi_1^* \ol \cH_\cX$ yields that
\begin{equation} \label{eq_siu_applied}
\adeg_1(g\circ f) \leqslant (d+1) \dfrac{\adeg \big( \widehat c_1 (u^* \pi_2^* \ol \cH_\cY)^d \cdot \widehat c_1 (v^* \pi_4^* \ol \cH_\cZ) \big)}{\adeg \big( \widehat c_1 (u^* \pi_2^* \ol \cH_\cY)^{d+1} \big)} \adeg_1(f). 
\end{equation}
Since $u^* \pi_2^* = v^* \pi_3^*$, we have:
\begin{equation*}
\dfrac{\adeg \big( \widehat c_1 (u^* \pi_2^* \ol \cH_\cY)^d \cdot \widehat c_1 (v^* \pi_4^* \ol \cH_\cZ) \big)}{\adeg \big( \widehat c_1 (u^* \pi_2^* \ol \cH_\cY)^{d+1} \big)} = \dfrac{\adeg \big( \widehat c_1 (v^* \pi_3^* \ol \cH_\cY)^d \cdot \widehat c_1 (v^* \pi_4^* \ol \cH_\cZ) \big)}{\adeg \big( \widehat c_1 (v^* \pi_3^* \ol \cH_\cY)^{d+1} \big)}.
\end{equation*}
Using the projection formula (cf.~\cite[Proposition 2.3.1(iv)] {bost_gillet_soule}), we obtain:
\begin{equation*}
\dfrac{\adeg \big( \widehat c_1 (v^* \pi_3^* \ol \cH_\cY)^d \cdot \widehat c_1 (v^* \pi_4^* \ol \cH_\cZ) \big)}{\adeg \big( \widehat c_1 (v^* \pi_3^* \ol \cH_\cY)^{d+1} \big)} = \dfrac{\adeg \big( \widehat c_1 (\pi_3^* \ol \cH_\cY)^d \cdot \widehat c_1 (\pi_4^* \ol \cH_\cZ) \big)}{\adeg \big( \widehat c_1 (\pi_3^* \ol \cH_\cY)^{d+1} \big)} = \dfrac{\adeg_1(g)}{\adeg \big( \widehat c_1(\ol \cH_\cY)^{d+1} \big)}.
\end{equation*}
The above equalities together with the inequality \eqref{eq_siu_applied} yields:
\begin{equation*}
\adeg_1(g\circ f) \leqslant C \adeg_1(f) \adeg_1(g),
\end{equation*}
where $C = (d+1) / \adeg \big( \widehat c_1(\ol \cH_\cY)^{d+1} \big)$, and the theorem is proved.
\end{proof}

\subsection{Proof of Theorem \ref{thmE}}

By Theorem~\ref{thmB} applied to $k=1$, we have
\[
\alpha_1(f) \geqslant \max \{ \lambda_1(f), \lambda_0(f) \} = \lambda_1(f).
\]
For the converse inequality, we apply Theorem~\ref{thm_upper_bound}.  There is a constant $C>0$ such that for any integer $p$, there exist some points $x_i$ such that $  \sum a_i [x_i] = {\pi_1}_* (Z_{1,\epsilon}) \in Z^n(X)$ and
\begin{equation*}
\adeg_1(f^p) \leqslant \sum a_i h_{\ol\cH}(f^p x_i) + C \deg_1(f^p).
\end{equation*}
By Matsuzawa's theorem (cf.~\cite[Theorem~1.4]{matsuzawa}), the growth rate of the height $h_{\ol\cH}(f^p x_i)$ of any point $x_i$ is bounded above by $\lambda_1(f)$, so we obtain:
\begin{equation*}
\alpha_1(f) \leqslant \lambda_1(f).
\end{equation*}
We finally conclude that $\alpha_1(f) = \lambda_1(f)$, as required.
\qed


\bibliographystyle{amsalpha}
\bibliography{ref_relative_arithmetic_degree}

\end{document}